\documentclass[a4paper,11pt]{amsart}

\usepackage{a4wide}

\usepackage{color}
\usepackage{tensor}
\usepackage{enumitem} 
\usepackage{amsfonts,amssymb,amsmath,amsthm,graphicx,latexsym}
\usepackage{mathrsfs,dsfont}

\newcommand{\bb}{\mathfrak{b}}

\newcommand{\id}{\iota} 
\DeclareMathOperator\linsp{span}

\newcommand{\qg}{\mathbb{G}}

\usepackage{dsfont}
\newcommand{\WW}{{\mathds{V}\!\!\text{\reflectbox{$\mathds{V}$}}}}
\newcommand{\Ww}{\mathds{W}}
\newcommand{\wW}{\text{\reflectbox{$\Ww$}}\:\!}
\newcommand{\ww}{\textup{W}}  

\newcommand\fs{B}
\newcommand\eb{E}

\newcommand{\otol}{\overline{\otimes}}

\newcommand\aphopf{\mathcal{AP}}
\newcommand\ap{\textit{AP}}

\theoremstyle{definition}
\newtheorem{dfn}{Definition}[section]

\theoremstyle{plain}
\newtheorem{thm}[dfn]{Theorem}
\newtheorem{lmma}[dfn]{Lemma}

\newtheorem{ppsn}[dfn]{Proposition}
\newtheorem{crlre}[dfn]{Corollary}

\theoremstyle{remark}

\newtheorem{rmrk}[dfn]{Remark}

\newcommand{\bdfn}{\begin{dfn}}
\newcommand{\bthm}{\begin{thm}}

\newcommand\wot{\mathop{\overline\otimes}}

\newcommand{\IC}{\mathbb{C}}

\newcommand{\IH}{\mathbb{H}}

\newcommand{\IN}{\mathbb{N}_0}

\newcommand{\IR}{\mathbb{R}}

\newcommand{\la}{\langle}
\newcommand{\ra}{\rangle}

\newcommand{\cls}{\mathcal{S}}

\newcommand{\Mone}{L^\infty(\qg)}
\newcommand{\Moned}{L^\infty(\widehat{\qg})}
\newcommand{\Mtwo}{L^\infty(\IR)}

\newcommand{\Hone}{L^2(\qg)}
\newcommand{\Htwo}{L^2(\IR)}

\newcommand{\qgdual}{{\widehat{\qg}}}

\newcommand{\ot}{\otimes}

\usepackage{comment}
\newcommand\K{\mathcal{K}}
\newcommand\cop{\Delta}

\DeclareMathOperator\Ker{Ker}
\DeclareMathOperator\Rep{Rep}
\newcommand\set[2]{\{\,#1:#2\,\}}

\newcommand\dual{\widehat}
\newcommand\conv{\star}
\newcommand\conj{\overline}
\newcommand\wt{\widetilde}
\newcommand\inv{^{-1}}
\newcommand\lone{L^1}
\newcommand\ltwo{L^2}
\newcommand\linf{L^\infty}

\newcommand{\tp}{\mathbin{\hbox{$\bigcirc$\hbox to
      0pt{\hspace{-0.81em}$\scriptstyle\top$\hfil}}}}

\begin{document}

\title[Admissibility Conjecture and Property (T)]%
  {Admissibility Conjecture and Kazhdan's Property (T) for quantum groups}
\author{Biswarup Das}
\address{Department of Mathematical Sciences, University of Oulu, Finland.}
\address{Instytut Matematyczny, Uniwersytet Wroc\l awski, Poland.}
\email{biswarup.das@math.uni.wroc.pl}
\author{Matthew Daws}
\address{Jeremiah Horrocks Institute, University of Central Lancashire, UK.}
\email{matt.daws@cantab.net}
\author{Pekka Salmi}
\address{Deprtment of Mathematical Sciences, University of Oulu, Finland.}
\email{pekka.salmi@oulu.fi}

\begin{abstract}
We give a partial solution to a long-standing open problem in the
theory of quantum groups, namely we prove that all finite-dimensional
representations of a wide class of locally compact quantum groups
factor through matrix quantum groups (Admissibility Conjecture for
quantum group representations). We use this to study Kazhdan's
Property (T) for quantum groups with non-trivial scaling group,
strengthening and generalising some of the earlier results obtained
by Fima, Kyed and So\l tan, Chen and Ng, Daws, Skalski and
Viselter, and Brannan and Kerr. Our main results are: 
\begin{enumerate}[label=\textup{(\roman*)}]
\item
  All finite-dimensional unitary representations of locally compact
  quantum groups which are either unimodular or arise through a
  special bicrossed product construction are admissible. 
 \item
   A generalisation of a theorem of Wang which characterises Property
   (T) in terms of isolation of finite-dimensional irreducible
   representations in the spectrum. 
 \item
   A very short proof of the fact that quantum groups with Property (T)
 are unimodular. 
 \item
 A generalisation of a quantum version of a theorem of Bekka--Valette
 proven earlier for quantum groups with trivial scaling group,
 which characterises Property  (T) in terms of non-existence of almost
 invariant vectors for weakly  mixing representations. 
 \item
 A generalisation of a quantum version of Kerr--Pichot theorem, proven
 earlier for quantum groups with trivial scaling group,
 which characterises Property (T) in terms of denseness properties of
 weakly mixing representations. 
\end{enumerate}
\end{abstract}

\maketitle

\section{Introduction}
Property (T) was introduced in the mid-1960s by Kazhdan, as a tool to
demonstrate that a large class of lattices are finitely generated. The
discovery of Property (T) was a cornerstone in group theory and the
last decade saw its importance in many different subjects like ergodic
theory, abstract harmonic analysis, operator algebras and some of the
very recent topics like C*-tensor categories (see
\cite{bekka_property_T, Connes_Weiss,
  popa_subfactors, neshveyev_Drinfeld_Center} and references
therein).
In the late 1980s the subject of operator algebraic quantum
groups gained prominence starting with the seminal work of Woronowicz
\cite{woroSUq2}, followed by works of Baaj, Skandalis,
\hbox{Woronowicz}, Van Daele, Kustermans, Vaes and others
\cite{baaj-skand,woro_mutiplicative,woro,
  kusvaes,masuda_woronowicz_nakagami}. Quantum 
groups can be looked upon as noncommutative analogues of locally
compact groups, so quite naturally the notion
of Property (T) appeared also in that more general context.
Property (T) was first studied within the framework of Kac
algebras (a precursor to the theory of locally compact quantum groups)
\cite{petrescu_property_T_for_Kac}, then for algebraic quantum groups
\cite{conti_property_T_algebraic_quantum_groups} and discrete quantum
groups \cite{Fima_property_T_discrete_quantum_group,kyed}, and more recently
for locally compact quantum groups
\cite{xiao_property_T,daws_skalski_viselter,brannan_property_T}.  

By definition a locally compact group $G$ has Property (T) if every
unitary representation with approximately invariant vectors
has in fact a non-zero invariant vector. This definition
extends verbatim to locally compact \emph{quantum} groups,
using the natural extensions of the necessary terms. 
By a result of Fima, a discrete quantum group having Property (T)
is necessarily a Kac algebra, which is equivalent to being unimodular
in the case of discrete quantum groups, and is the dual of a compact
matrix quantum group
\cite[Propositions~7~\&~8]{Fima_property_T_discrete_quantum_group}.
This is a quantum generalisation of a result originally due to Kazhdan
\cite[Theorem 1.3.1 \& Corollary~1.3.6]{bekka_property_T}.
In particular, while studying Property (T) for discrete quantum
groups, one is lead to consider only unimodular discrete quantum
groups. Since a discrete quantum group is
unimodular if and only if it is of Kac type, unimodular discrete
quantum groups have trivial scaling automorphism groups,
and this is important in what follows.

Generalising to locally compact quantum groups,
Brannan and Kerr proved that a second countable
locally compact quantum group with Property~(T) is necessarily unimodular
\cite[Theorem~6.3]{brannan_property_T} --
a result for which we will also give a new and short proof
without the second countability assumption.
So again while studying Property (T) for quantum groups, one is lead
to consider only unimodular quantum groups.
However, a \emph{unimodular} locally compact quantum group can
have a \emph{non-trivial scaling automorphism group}. Examples of
such locally compact quantum groups are Drinfeld doubles
  of non-Kac-type compact quantum groups: see Section
\ref{sec:quantum-double}.
A recent result of Arano (see
\cite[Theorem 7.5]{arano1}), which finds applications in the study of
C*-tensor categories and subfactors \cite{neshveyev_Drinfeld_Center,
  popa_subfactors}, states that the Drinfeld double of the
Woronowicz compact quantum group $SU_q(2n+1)$ has Property (T). This
produces a concrete example of a unimodular locally compact quantum
group with non-trivial scaling automorphism group, which has
Property~(T).

In this paper, we study Property (T) and related problems,
in particular on unimodular locally compact quantum groups
with non-trivial scaling automorphism group.
To enable this study, we prove the
`Admissibility Conjecture' for unimodular locally compact quantum
groups, that is, we show that every finite-dimensional
unitary representation of a unimodular locally compact quantum group
is admissible. The Admissibility Conjecture is a long-standing open
problem in the theory of quantum groups, which was implicitly stated
in \cite{soltan} and was conjectured in
\cite[Conjecture~7.2]{daws}. Admissibility of a finite-dimensional  
unitary representation of a quantum group means effectively that it
`factors' through a compact matrix quantum group.

Returning to locally compact groups,
we note the following important characterisation of
Property (T) by Bekka and Valette
\cite[Theorem~1]{bekka_property_T_amenable_representations}:
a locally compact group $G$ has Property (T) if and only if 
every unitary representation of $G$ with approximately invariant vectors
is not weakly mixing (i.e.\ admits a non-zero finite-dimensional
subrepresentation).
This characterisation turns out to be more useful from the application
perspective than the definition itself, as has been elucidated in
\cite{brannan_property_T}. An important consequence of the
Bekka--Valette theorem is the Kerr--Pichot theorem which states that if
$G$ does not have Property (T), then within the set of all unitary
representations on a fixed separable Hilbert space, the weakly mixing
ones form a dense $G_\delta$-set in the weak topology, strengthening
an earlier result of Glasner and Weiss
\cite[Theorem~2$^\prime$]{glasner_property_T} concerning
the density of ergodic representations. Another important result along
characterising 
Property (T) is that $G$ has Property (T) if and only if the trivial
representation is isolated in the hull--kernel topology
of the dual space $\widehat{G}$ \cite{wang_property_T}. A theorem of
Wang \cite[Theorem~2.1]{wang_property_T} (see also \cite[Theorem
  1.2.5]{bekka_property_T}) extends this to all irreducible
finite-dimensional  unitary representations of $G$ i.e. $G$ has
Property (T) if and only if all irreducible finite-dimensional
unitary representations of $G$ are isolated in $\widehat{G}$. This in
particular helps us better understand the structure of the
full group C*-algebra $C^*_u(G)$ and has other important applications
\cite[Chapter I]{bekka_property_T}.  

The first quantum version of Wang's characterisation of Property (T)
was proven for discrete quantum groups (see \cite[Remark 5.4]{kyed}). 
Under the additional hypothesis of having low duals,
Bekka--Valette and Kerr--Pichot theorems were proven for unimodular
discrete quantum groups
\cite[Theorem 7.3, 7.6 \&  9.3]{daws_skalski_viselter}.
Recall that for discrete quantum groups being unimodular is the
same as being a Kac algebra
and that Kac algebras form a class of locally compact quantum groups
with trivial scaling group. 
The study of Property (T) on quantum groups
progressed along these lines with 
the quantum versions of the theorems of Wang,
Bekka--Valette and Kerr--Pichot
generalised to quantum groups with trivial scaling groups
in \cite[Proposition 3.2 \& Theorem 3.6]{xiao_property_T} and in
\cite[Theorem 4.7, 4.8, 4.9 \& 5.1]{brannan_property_T}.  

Upon giving an affirmative answer to the Admissibility Conjecture for
unimodular locally compact quantum groups (including those with
non-trivial scaling groups), we proceed to prove a quantum version of
Wang's  theorem for them as well as generalised versions of
the Bekka--Valette and the Kerr--Pichot theorems.
In particular, we show that for unimodular quantum groups with
non-trivial scaling automorphism group, the weakly mixing
representations are dense in the set of representations on a separable
Hilbert space if the quantum group does not have Property (T).

\subsection*{Acknowledgement}

We are grateful to Adam Skalski for various mathematical discussions. All the authors would also like to thank Ami Viselter for various mathematical comments on an earlier version of the paper. The first author gratefully acknowledges
the support of the Mathematical Research Unit at the University of
Oulu, Finland, during the years 2015-2017 and partial support of the
Simons Foundation grant 346300 and the Polish Government MNiSW
2015--2019 matching fund.

\section{Notation and terminology} 
\label{Subsection: LCQG}

We collect a few facts from the theory of locally compact quantum
groups, as developed in the papers
\cite{kus,kusvaes,kusvaes_vNa}, and we refer the reader to
\cite{kuster_notes_LCQG} for a summary of the main results in the
theory. We will take the viewpoint 
that whenever we consider a locally compact quantum group, the symbol
$\qg$ denotes the underlying `locally compact quantum space' of the
quantum group. From this viewpoint, for a locally compact quantum
group $\qg$ the corresponding C*-algebra of `continuous functions on
$\qg$ vanishing at infinity' will be denoted by $C_0(\qg)$. It is
equipped with a coassociative \emph{comultiplication}
$\cop:C_0(\qg)\to M(C_0(\qg)\ot C_0(\qg))$ and left and right Haar
weights $\phi$ and $\psi$ \cite[Definition~4.1]{kusvaes}
(where we use the notation that $M(A)$ denotes the multiplier algebra
of a C*-algebra $A$). An important aspect of the theory of locally
compact quantum groups is  a noncommutative Pontryagin duality theory,
which in particular allows one to view both a locally compact group
and its `dual' as locally 
compact quantum groups \cite[Subsection 6.2]{kuster_notes_LCQG},
\cite[Section 8]{kusvaes}. The dual of $\qg$, which is again a
locally compact quantum group, is denoted by $\qgdual$.
(For example if $\qg=G$, a locally compact group, then
$C_0(\qgdual)=C^\ast_r(G)$, the reduced group C*-algebra of $G$.)
As in the case of $\qg$,
$C_0(\qgdual)$ is equipped with a coassociative comultiplication
$\widehat{\cop}:C_0(\qgdual)\to M(C_0(\qgdual)\ot C_0(\qgdual))$
and left and right Haar weights $\widehat{\phi}$ and $\widehat{\psi}$.
By the definition of the dual quantum group as given in \cite[Definition
  8.1]{kusvaes}, we may think of both the C*-algebras $C_0(\qg)$ and
$C_0(\qgdual)$ as acting faithfully and non-degenerately on the
Hilbert space $L^2(\qg)$ (obtained by applying the GNS construction
to the left Haar weight $\phi$).
A locally compact quantum group is said to be \emph{compact} if
$C_0(\qg)$ is unital. Compact quantum groups themselves 
have a very nice theory \cite{woro, vandaele_cqg}.

The fundamental multiplicative unitary
$\ww\in M(C_0(\qg)\ot C_0(\qgdual))$ (called the \emph{Kac--Takesaki operator}
in the \emph{theory of Kac algebras} \cite{enock_Kac}, a precursor to
the theory of locally compact quantum groups) implements the
comultiplications as follows: 
\[
\cop(x)=\ww^\ast(1\ot x)\ww,\qquad x\in C_0(\qg),
\]
and
\[
\widehat{\Delta}(x)=\chi(\ww(x\ot1)\ww^\ast),\qquad x\in C_0(\qgdual),
\]
where $\chi:B(L^2(\qg)\ot L^2(\qg))\to B(L^2(\qg)\ot L^2(\qg))$ is
the flip map
\cite[Definition 6.12 \& Subsection 6.2]{kuster_notes_LCQG},
\cite[pp. 872--873, Definition 8.1]{kusvaes}.

The von Neumann algebra generated by $C_0(\qg)$
(respectively, by $C_0(\qgdual)$) in $B(L^2(\qg))$ will be
denoted by $L^\infty(\qg)$ (respectively, by $L^\infty(\qgdual)$).
Then the preduals of $L^\infty(\qg)$ and $L^\infty(\qgdual)$ are
denoted by $L^1(\qg)$ and $L^1(\qgdual)$, respectively. The above formulas
imply that both the maps $\cop$ and $\widehat{\cop}$ can be
lifted to normal $*$-homomorphisms on $L^\infty(\qg)$ and
$L^\infty(\qgdual)$. The preadjoints of the normal maps $\cop$ and
$\widehat{\cop}$ equip $L^1(\qg)$ and $L^1(\qgdual)$ with the
structure of a completely contractive Banach algebra.
The universal C*-algebra $C_0^u(\qg)$ associated with $\qg$
is the universal C*-envelope of a distinguished Banach $*$-algebra
$L^1_\sharp(\qgdual)$ (as an algebra,
$L^1_\sharp(\qgdual)\subset L^1(\qgdual)$). 
The C*-algebra $C^u_0(\qg)$ is equipped with a coassociative
comultiplication denoted by $\cop_u:C^u_0(\qg)\to M(C^u_0(\qg)\ot C^u_0(\qg))$
\cite[Proposition 6.1]{kus}. Moreover,
there exists a  surjective $*$-homomorphism
$\Lambda_\qg:C^u_0(\qg)\to C_0(\qg)$
called the \emph{reducing morphism}, which intertwines the comultiplications:
$(\Lambda_\qg\ot\Lambda_\qg)\circ\cop_u=\cop\circ\Lambda_\qg$.
The comultiplication on $C_0^u(\qgdual)$ is denoted by
$\widehat{\cop}_u$.
The dual space $C_0^u(\qg)^*$ is a completely contractive Banach
algebra with respect to the \emph{convolution} product
\[
\omega_1\conv \omega_2 = (\omega_1\ot \omega_2)\circ \cop_u,
\qquad \omega_1, \omega_2\in C_0^u(\qg)^*.
\]

As shown in \cite[Corollary 4.3 and Proposition 5.2]{kus}, the
fundamental multiplicative unitary $\ww$ admits a lift
$\wW\in M(C_0(\qg)\ot C^u_0(\qgdual))$ called the semi-universal bicharacter
of $\qg$. It is characterised by the following universal property:  
there is a one-to-one correspondence between
\begin{itemize}
 \item 
   unitary elements $U\in M(C_0(\qg)\ot B)$ such that  
   $(\cop\ot\iota)(U)=U_{13}U_{23}$ (here $B$ is a C*-algebra)
 \item
 non-degenerate $*$-homomorphisms $\pi_U:C^u_0(\qgdual)\to M(B)$
 satisfying $(\iota\ot\pi_U)(\wW)=U$. 
\end{itemize}

There are two important maps associated with a locally compact
quantum group $\qg$ related to the inverse operation of a group.
The \emph{antipode} $S$ is a densely defined norm-closed map on
$C_0(\qg)$ \cite[Section 5]{kusvaes}. It can be extended to a
densely defined strictly closed map on $M(C_0(\qg))$
\cite[Remark 5.44]{kusvaes}. The antipode has a universal counterpart
$S_u$ which is a densely defined map on $C^u_0(\qg)$
\cite[Section 9]{kus}. The \emph{unitary antipode}
$R:C_0(\qg)\to C_0(\qg)$ is a $*$-antiautomorphism
\cite[Proposition~5.20]{kusvaes} satisfying   
$(R\ot R)\circ\cop=\chi\circ\cop\circ R$.
Its universal counterpart $R_u$ is a $*$-antiautomorphism of
$C^u_0(\qg)$ having similar properties as $R$ \cite[Proposition 7.2]{kus}. 
The corresponding maps on the dual quantum group
are denoted by $\widehat{S}$, $\widehat{S}_u$, $\widehat{R}$ and
$\widehat{R}_u$. It is worthwhile to note that if $\qg$ is a Kac
algebra, then $S=R$ and $S_u=R_u$.
In general, the antipode has a polar decomposition
$S = R\circ \tau_{-i/2}$ where $\tau_{i/2}$ is defined by 
an analytic extension of the scaling automorphism group
$(\tau_t)_{t\in\IR}$,
where each $\tau_t: C_0(\qg) \to C_0(\qg)$ is a $*$-automorphism.
The scaling group is implemented by the modular operator $\widehat{\nabla}$
of the dual left Haar weight.
(If $\qg$ is a Kac algebra, $\widehat{\nabla}$ is affiliated to the center of $L^\infty(\qg)$ and consequently $\tau_t = \id$ for every
$t\in\IR$.)
The scaling group of the dual quantum group $\qgdual$ is denoted
by $\widehat{\tau}$. The scaling groups have
their universal counterparts on the C*-algebras
$C^u_0(\qg)$ and $C^u_0(\qgdual)$ and these are denoted by 
$\tau^u$ and $\widehat{\tau}^u$, respectively (see \cite[Definition 4.1]{kus}). 
The universal antipode has a similar decomposition
$S_u = R^u\circ \tau^u_{-i/2}$.

The modular automorphism group associated to the left Haar weight
$\phi$ on $\qg$ is denoted by $(\sigma_t)_{t\in\IR}$,
and its universal counterpart by $(\sigma^u_t)_{t\in\IR}$.

As shown in \cite[Proposition 6.3]{kus} there exist \emph{counits}
$\epsilon_u:C^u_0(\qg)\to\IC$ and
$\widehat{\epsilon}_u:C^u_0(\qgdual)\to\IC$,
which are $*$-homomorphisms satisfying  
\[
(\epsilon_u\ot\iota)(\cop_u(x)) = x
=(\iota\ot\epsilon_u)(\cop_u(x)), \qquad x\in C^u_0(\qg),
\]
and 
\[
(\widehat{\epsilon}_u\ot\iota)(\widehat{\cop}_u(x)) = x
=(\iota\ot\widehat{\epsilon}_u)(\widehat{\cop}_u(x)),
\qquad x\in C^u_0(\qgdual). 
\]
Moreover, $(\iota\ot\widehat{\epsilon}_u)(\wW)=1$.

A \emph{representation} of a locally compact quantum group $\qg$
on a Hilbert space $H$ is an invertible
element $U \in M(C_0(\qg)\ot \K(H))$ such that
\begin{equation} \label{eq:corep}
(\cop\ot \id)(U) = U_{13}U_{23}.
\end{equation}
We are mostly interested in \emph{unitary representations} in which case
$U$ is further a unitary.
Note that if $U\in \linf(\qg)\otol B(H)$ is a unitary
that satisfies \eqref{eq:corep}, then $U \in M(C_0(\qg)\ot \K(H))$.
Indeed, \eqref{eq:corep} implies that
\[
U_{13} = \ww^*_{12}U_{23}\ww_{12}U_{23}^* \in
M\bigl(C_0(\qg)\ot\K(L^2(\qg))\ot\K(H)\bigr)
\]
as $\ww\in M\bigl(C_0(\qg)\ot\K(L^2(\qg))\bigr)$, and the
claim follows.

The \emph{trivial representation}
$1\ot 1\in M(C_0(\qg)\ot \IC)$ is denoted by $1$. 
Two representations $U$ and $V$ are \emph{similar} if
there is an invertible $a\in B(H_V, H_U)$
such that $V = (1\ot a^{-1}) U (1\ot a)$
(where $B(H_V, H_U)$ denotes the set of bounded linear maps from $H_V$
to $H_U$). If $a$ is further an unitary
map $U$ and $V$ are said to be (\emph{unitarily}) \emph{equivalent}.
Given a representation $U\in M(C_0(\qg)\ot \K(H))$,
its \emph{contragradient representation}
is
\[
U^c = (R\ot \top) U \in M(C_0(\qg)\ot \K(\conj{H}))
\]
where $R$ is the unitary antipode,
$\top(x)\conj\xi = \conj{x^*\xi}$ for $x\in B(H)$
and $\xi \in H$ and $\conj{H}$ is the dual Hilbert space of $H$.

If $U\in M(C_0(\qg)\ot \K(H))$ and $V\in M(C_0(\qg)\ot \K(K))$
are representations of $\qg$, their tensor product is 
\[
U\tp V = U_{12}V_{13}\in M(C_0(\qg)\ot\K(H\ot K)).
\]
As noted above, every unitary representation $U$ of $\qg$ is
associated with a representation $\pi$ of the C*-algebra
$C_0^u(\qgdual)$ via $U = (\id\ot \pi)\wW$, and vice versa.
In particular, the trivial representation is associated with
the counit $\widehat\epsilon_u$.
If $\pi_U$ and $\pi_V$ are the representations of $C_0^u(\qgdual)$
associated with $U$ and $V$, respectively, then
the representation $(\pi_U\ot\pi_V)\circ\chi\circ\widehat\cop_u$
is associated with $U\tp V$.

\section{A characterisation of admissible finite-dimensional unitary
   representations} 
\label{Section: Admissibility of a finite-dimensional unitary
  representation of a LCQG}  

Let $U\in M(C_0(\qg)\ot M_n(\IC))$ be a finite-dimensional unitary
representation of a locally compact quantum group $\qg$. Choosing the
standard basis for $\IC^n$ we write $U=(U_{ij})_{i,j=1}^n$, where $U_{ij}\in M(C_0(\qg))$
for $i,j=1$,~$2$, \ldots, $n$ are the matrix coefficients of $U$, and
we have $\cop(U_{ij})=\sum_{k=1}^nU_{ik}\ot U_{kj}$ for
$i,j=1,2,\ldots, n$. 

\begin{dfn} \label{Definition: Admissible representations of LCQG}
A finite-dimensional unitary representation $U=(U_{ij})_{i,j=1}^n$ of a locally
compact quantum group $\qg$ is called \emph{admissible} if
$U^t:=(U_{ji})_{i,j=1}^n$ is invertible in the C*-algebra
$M(C_0(\qg)\ot M_n(\IC))$.  
\end{dfn}

Admissible finite-dimensional representations of locally compact
quantum groups first appeared in the work of So\l tan 
\cite{soltan}, who introduced the quantum Bohr compactification of a
locally compact quantum group.
Daws \cite{daws} studied further the quantum Bohr compactification
as well as questions related to admissibility.
It was conjectured (see \cite[Conjecture 7.2]{daws})
that  every finite-dimensional unitary representation
of a locally compact quantum group is admissible.
Note that this conjecture is already false if we
replace quantum group by quantum semigroup: a
counterexample due to Woronowicz is given in
\cite[Example 4.1]{wang_free_product_CQG}.

From the results in \cite{woromatrixQG} it follows that if
$U\in M(C_0(\qg)\ot M_n(\IC))$ is an admissible finite-dimensional unitary
representation, then the C*-algebra generated by the matrix
coefficients of $U$ in $M(C_0(\qg))$ gives rise to  a compact matrix quantum
group. It follows that a finite-dimensional unitary representation
of a locally compact quantum group is admissible if and only if
it factors, in this sense,
through a compact quantum group (as finite-dimensional
representations of compact quantum groups are admissible).
It is worthwhile to note that the use of $C_0(\qg)$ above is
purely a matter of convenience: we can do similar considerations for
$C^u_0(\qg)$ as well. 

The linear span of all matrix coefficients of admissible unitary
representations of $\qg$ is denoted by $\aphopf(\qg)$.
Note that $\aphopf(\qg)$ is a Hopf $*$-algebra. Its norm closure
in $M(C_0(\qg))$ is denoted by $\ap(\qg)$.
It may be that $\ap(\qg)$ is not the universal C*-completion of
$\aphopf(\qg)$. The compact quantum group
associated with $\aphopf(\qg)$ is the \emph{quantum Bohr compactification}
of $\qg$ and is denoted by $\bb\qg$. See \cite{soltan, daws}
for more details. 

Next we characterise the admissibility of a finite-dimensional
unitary representation in terms of the
scaling group.

\begin{ppsn}\label{Proposition: covariant finite-dimensional C* representations give admissible corepresentations}
Let $\pi:C^u_0(\qgdual)\to M_n(\IC)$ be a non-degenerate
$*$-homomorphism, $U=(\iota\ot\pi)(\wW)$ and
$L_U = \linsp\set{U_{ij}}{i,j=1,2,\ldots, n}$. Then the
following statements are equivalent: 
\begin{enumerate}[label=\textup{(\roman*)}]
\item 
The representation $U\in M(C_0(\qg)\ot M_n(\IC))$ is admissible.
\item
The vector space $L_U$ is invariant under the scaling group $(\tau_t)$. 
\item
  There exists a strongly continuous one-parameter automorphism group
  $(\alpha_t)$ on $M_n(\IC)$ such that 
  \[
  \pi\circ\widehat{\tau}^u_t=\alpha_t\circ\pi\qquad \text{for every }~t\in\IR. 
  \]
\end{enumerate}
\end{ppsn}

\begin{proof}
(i)$\implies$(iii):
Since $U=(U_{ij})$ is admissible, 
$U_{ij}\in \aphopf(\qg)$ for all $i, j=1, 2, \ldots, n$.
Let $\Theta:C^u(\bb\qg)\to \ap(\qg)\subset M(C_0(\qg))$
be the canonical quantum group morphism from the universal
C*-completion of $\aphopf(\qg)$ onto the closure of $\aphopf(\qg)$. 
Let $\widehat{\Theta}:C_0^u(\qgdual)\to M(C^u_0(\widehat{\bb\qg}))$
be the dual morphism so that
$(\id\ot\widehat\Theta)\wW = (\Theta\ot\id)\WW_{\bb\qg}$
(see \cite[Corollary 4.3]{mrw}).
Since $\widehat{\bb\qg}$ is a discrete quantum group, we can drop $u$ from the
notation and write $c_0(\widehat{\bb\qg})$ for notational convenience.

Let $\wt U\in M(C^u(\bb\qg)\ot M_n(\IC))$
be the lift of the representation $U$,
i.e.\ $(\Theta\ot\id)\wt U = U$. 
Then there is a non-degenerate $*$-homomorphism
$\phi:c_0(\widehat{\bb\qg})\to M_n(\IC)$
such that $\wt U = (\id\ot\phi)\WW_{\bb\qg}$.
We have
\[
(\id\ot\phi\circ\widehat\Theta)\wW
= (\Theta\ot\phi)\WW_{\bb\qg} = U = (\id\ot \pi)\wW,
\]
which implies that $\pi=\phi\circ\widehat{\Theta}$.

There is an unbounded strictly positive operator $K$
affiliated to the von Neumann algebra $\ell^\infty(\widehat{\bb\qg})$
such that $K$ implements the scaling group $(\tau'_t)$
of $\widehat{\bb\qg}$ in the sense that 
$\tau'_t(x) = K^{-2it} x K^{2it}$ for
every $x\in\ell^\infty(\widehat{\bb\qg})$ and $t\in\IR$  \cite[Proposition 4.3]{discrete_QG_daele}.
Note that  $K^{2it}\in \ell^\infty(\widehat{\bb\qg})$ for all $t\in\IR$. 

Now for $x\in C_0^u(\qgdual)$ and $t\in\IR$, 
\[
\pi(\widehat{\tau}^u_t(x))
= \phi(\widehat{\Theta}(\widehat{\tau}^u_t(x)))
= \phi(\tau^\prime_t(\widehat{\Theta}(x)))
\]
since $\widehat{\Theta}$ intertwines the scaling groups
as a morphism of quantum groups (by \cite[Remark 12.1]{kus}).
Continuing the calculation, we have
\[
\pi(\widehat{\tau}^u_t(x))
= \phi\bigl(K^{-2it}\widehat{\Theta}(x)K^{2it}\bigr)
= \phi(K^{-2it})\pi(x)\phi(K^{2it}). 
\]
Defining $\alpha_t(A) =  \phi(K^{-2it}) A\phi(K^{2it})$ for $t\in\IR$
and $A\in M_n(\IC)$ yields the desired result.

(iii)$\implies$(ii):
From the version of \cite[Proposition 9.1]{kus} for $C^u_0(\qgdual)$,
we have $(\tau_t\ot\iota)(\wW) = (\iota\ot\widehat{\tau}^u_{t})\wW$
for all $t\in\IR$. Then 
\begin{equation*}
\begin{split}
(\tau_t\ot\id)(U) &= (\tau_t\ot\id)\circ (\id \ot\pi)(\wW)
=(\id\ot \pi\circ\widehat{\tau}^u_t)(\wW)\\
&=(\id\ot \alpha_t\circ\pi)(\wW) = (\id\ot \alpha_t)(U).
\end{split}
\end{equation*}
It follows that $L_U$ is invariant under the scaling group $(\tau_t)$  

(ii)$\implies$(i):
Let $\tau_z$ denote the
analytic extension of $(\tau_t)_{t\in\IR}$ at $z\in\IC$ in the
$\sigma$-weak topology on $L^\infty(\qg)$.  For $x\in L_U$ define in a standard
way the smear of $x$
\[ x_n = \frac{n}{\sqrt\pi} \int_{\mathbb R}
e^{-n^2t^2} \tau_t(x) \ dt, \]
where the integral converges in the $\sigma$-weak topology.
Each $x_n$ is analytic for $(\tau_t)$.
As $L_U$ is invariant under the scaling group and is finite-dimensional,
it follows that $x_n \in L_U$ for each $n$.  It follows that $L_U \subseteq
D(\tau_z)$ for all $z$.  In particular, $L_U\subset D(\tau_{{\frac{i}{2}}})
= D(S\inv)$, due to the polar
decomposition of $S\inv$. It then follows that
$U_{ij}^\ast\in D(S)$ for all $i,j=1,2,\ldots, n$. By
\cite[Proposition 3.11]{daws}, $U$ is admissible.
\end{proof}

\begin{rmrk}\label{Remark: Admissiblity conjecture is true for Kac algebras}
It follows from Proposition \ref{Proposition: covariant
  finite-dimensional C* representations give admissible
  corepresentations} that for quantum groups with trivial scaling
automorphism group, all finite-dimensional unitary representations are
admissible. In particular the Admissibility Conjecture
holds for Kac algebras, as also noted in \cite{daws}.
\end{rmrk}

\begin{rmrk}\label{Remark: Admissiblity conjecture implies
    representation is covariant} 
For a general locally compact quantum group, it follows that
admissible representations of $\qg$
correspond to those finite-dimensional representations of
$C^u_0(\qgdual)$ which are covariant with respect to the scaling
action of $\IR$ on $C^u_0(\qgdual)$. 
\end{rmrk}

\begin{rmrk}\label{Remark: CQG}
Let $U\in M(C_0(\qg)\ot M_n(\IC))$ be admissible.  Then $U$ is a
corepresentation of the compact quantum group $AP(\qg)$. As the
inclusion $AP(\qg) \rightarrow M(C_0(\qg))$ intertwines $R$, we see
that $U^c$ may also be considered as a corepresentation of the compact
quantum group $AP(\qg)$. Combining \cite[Definition 1.3.8, Definition
  1.4.5 and equation (1.7.1)]{nesh-tuset} we can conclude that
$\overline{U} = (U_{ij}^*)_{i,j=1}^n$ is equivalent to $U^c$, as
corepresentations of $AP(\qg)$, and hence also of $C_0(\qg)$ (it is
worthwhile to point out that in \cite{nesh-tuset}, $U^c$ and
$\overline{U}$ are what we call $\overline{U}$ and $U^c$ respectively
in this article).  

%
\end{rmrk}

\section{Examples of locally compact quantum groups with admissible
  representations}

It is an open question whether all
finite-dimensional unitary representations of a locally compact
quantum group are admissible. In this subsection we give
examples of locally compact quantum groups for which the
Admissibility Conjecture is true.

\subsection{Quantum groups arising from a bicrossed product
  construction}
\label{Subsection: Examples of locally compact
  quantum groups with admissible representations} 
We recall some facts about the bicrossed product construction of a matched
pair of quantum groups and refer to \cite{bicrossed} for
details. Define a normal $*$-homomorphism $\tau$ as follows: 
\[
\tau : L^\infty(\qg)\otol L^\infty(\IR)\to L^\infty(\qg)\otol L^\infty(\IR),\quad
\tau(f)(t)=\tau_t(f(t)) \quad (f\in L^\infty(\qg)\otol L^\infty(\IR),~t\in\IR) 
\]
where we have identified $L^\infty(\qg)\otol L^\infty(\IR)$ with
$L^\infty(\IR,L^\infty(\qg))$.
Then the map $\tau$ is a matching between the locally compact quantum groups
$\qg$  and $\IR$ with trivial cocycles
($\mathscr{U}=1$ and $\mathscr{V} = 1$),
in the sense of \cite[Definition 2.1]{bicrossed}. 

  In the notation of \cite[Definition 2.1]{bicrossed}, we define
  a left action
  $\alpha:L^\infty(\IR)\to L^\infty(\qg)\otol L^\infty(\IR)$
  and a right action
  $\beta:L^\infty(\qg)\to L^\infty(\qg)\otol L^\infty(\IR)$ by the
  formulas  
\[
\alpha(f)=\tau(1\ot f)=1 \ot f,\qquad f\in L^\infty(\IR), 
\]
and
\[
\beta(x)(t) = \tau(x\ot 1)(t) = \tau_t(x),\qquad t\in\IR,~x\in L^\infty(\qg). 
\]
Consider the crossed products
$M:= \qg\tensor[_\alpha]{\ltimes}{} L^\infty(\IR)$ and
$\widehat{M}:=L^\infty(\qg)\tensor[]{\rtimes}{_\beta} \IR$. 
It follows from the discussion in \cite[Subsection 2.2]{bicrossed} 
that $M$ is a locally compact  quantum group in the reduced form
and $\widehat{M}$ is the reduced dual. Note that since the left
action $\alpha$ is trivial, $M = L^\infty(\qgdual)\wot L^\infty(\IR)$.

Denote the quantum group underlying $\widehat{M}$ by
$\qg\rtimes_{\beta}\IR$.
The left multiplicative unitary of $\qg\rtimes_{\beta}\IR$
is 
\[
\ww = \ww^\IR_{24} \bigl((\id\ot\beta)(\dual\ww^\qg)\bigr)_{134}, 
\]
where the leg numbering refers to the underlying Hilbert
space $\ltwo(\qg)\ot\ltwo(\IR)\ot\ltwo(\qg)\ot\ltwo(\IR)$
(see \cite[p. 141]{double_crossed_product}).
It follows from the above formula that 
$C_0(\qg\rtimes_{\beta}\IR) = C_0(\qg)\rtimes_{\beta}\IR$.
We will show a similar characterisation of 
$C^u_0(\qg\rtimes_{\beta}\IR)$. Consider the action
\[
\beta^u:C^u_0(\qg)\to M(C^u_0(\qg)\ot C_0(\IR)),\qquad
 \beta^u(x)(t)=\tau^u_t(x), \quad x\in C^u_0(\qg), t\in\IR.
\]

\begin{lmma}
$C_0^u(\qg\rtimes_\beta \IR) \cong C_0^u(\qg)\rtimes_{\beta^u}\IR$. 
\end{lmma}

\begin{proof}
  Denote the comultiplication of $\qg\ltimes_\alpha\IR$ by $\Delta$.
  We recall the following formula from \cite[p. 141]{double_crossed_product}: 
  \[
  (\iota\ot\Delta)({\ww^\qg}\ot 1_{\Mtwo}) =
  \ww^\qg_{14}\bigl((\iota\ot\alpha)\circ\beta\ot\iota\bigr)({\ww^\qg})_{1452},  
  \]
  where the leg numbering is done with respect to the Hilbert space
  $\Hone\ot\Hone\ot\Htwo\ot\Hone\ot\Htwo$.
  Applying the definition of $\alpha$, we obtain 
    \begin{equation}\label{eq:first-rep}
      (\iota\ot\Delta)({\ww^\qg}\ot 1_{\Mtwo}) = \ww^\qg_{14}
      \bigl((\iota\ot\chi)\circ(\beta\ot\iota)({\ww^\qg})\bigr)_{125}.
    \end{equation}
    where $\chi:\Mtwo\otol\Moned\to\Moned\otol\Mtwo$ is the flip map.
    Write $U = (\iota\ot\chi)\circ(\beta^u\ot\iota)({\Ww^\qg})$.
    Then \eqref{eq:first-rep} says that
    \[
    (\Lambda_\qg\ot \id\ot\id\ot\id\ot\id)\bigl((\iota\ot\Delta)({\Ww^\qg}\ot
    1_{\Mtwo})\bigr) = (\Lambda_\qg\ot \id\ot\id\ot\id\ot\id)(\Ww^\qg_{14} U_{125}).
    \]
    where $\Lambda_\qg:C^u_0(\qg)\to C_0(\qg)$ is the reducing
    morphism.   We would like to apply \cite[Result 6.1]{kus} (or a
    version of it for $\qg$) to deduce that in fact
    \begin{equation}
      \label{eq:univ-rep}
      (\iota\ot\Delta)(\Ww^\qg\ot1_{\Mtwo})=\Ww^\qg_{14}U_{125}.
    \end{equation}
    To this end, we need to check that both sides of the preceding
    equation define corepresentations of $C_0^u(\qg)$. 

   Since $(\tau_t^u\ot\tau_t^u)\circ\Delta_u=\Delta_u\circ\tau_t^u$ for all
   $t\in\IR$, it follows that $U\in M(C_0^u(\qg)\ot \K(\Hone\ot\Htwo)$ is a
   corepresentation of $C_0^u(\qg)$.
   Therefore, $\Ww^\qg\tp U$ is a corepresentation and hence
   also 
   \[
   \bigl((\id\ot\chi\ot\id)(\Ww^\qg\tp U)\bigr)_{1245}
   \]
   where $\chi$ is the appropriate flip map.
   It follows that the right-hand side of
   \eqref{eq:univ-rep} is a 
   corepresentation, and the left-hand side clearly is as well.

Suppose that $X\in M(C_0(\qg \ltimes_\alpha\IR)\ot \K(H))$ is a
unitary representation of $\qg\ltimes_\alpha\IR$.
By \cite[Proposition 4.1]{double_crossed_product} there
are unitary representations $z\in M(C_0(\qgdual)\ot \K(H))$ and
$y\in M(C_0(\IR)\ot \K(H))$
such that $X=(\alpha\ot\iota)(y)z_{13}$, where the leg numbering
is done with respect to the Hilbert space $\Hone\ot\Htwo\ot H$.
(The unitarity of the representations is implicit in
the proof of \cite[Proposition 4.1]{double_crossed_product}.)
The equation after equation (4.2) in page 146 of
\cite{double_crossed_product} says that
\[
(\Delta\ot\iota)(z_{13})z^\ast_{35}
=(\alpha\ot\iota)(y^\ast)_{345}z_{15}(\alpha\ot\iota)(y)_{345}, 
\] 
where the leg numbering is done with respect to the Hilbert space
\[
\Hone\ot\Htwo\ot\Hone\ot\Htwo\ot H.
\]
Applying the fact that $\alpha(f)=1_{\Mone}\ot f$ for all $f\in\Mtwo$,
it follows that the above equation can be reduced to: 
\begin{equation}\label{Equation: equation 3}
(\Delta\ot\iota)(z_{13})z^\ast_{35}=y^\ast_{45}z_{15}y_{45}.
\end{equation}
Let $\pi: C^u_0(\qg)\to B(H)$ be the non-degenerate
C*-representation associated with $z$, so that we have
$(\iota\ot\pi)(\widehat{\wW^\qg})=z$. Then it is easy to check that  
\begin{equation*}
\begin{split}
  (\Delta\ot\iota)(z_{13})=
  (\iota\ot\iota\ot\iota\ot\iota\ot\pi)
  \bigl(\sigma\bigl((\id\ot \Delta)({\Ww^\qg}^\ast\ot1_{\Mtwo})\bigl)\bigr), 
\end{split}
\end{equation*}
where 
\[
\sigma:
M(C^u_0(\qg)\ot C_0(\qgdual)\ot C_0(\IR)\ot C_0(\qgdual)\ot C_0(\IR))
\to M(C_0(\qgdual)\ot C_0(\IR)\ot C_0(\qgdual)\ot C_0(\IR)\ot C^u_0(\qg))
\]
permutes the coordinates according to the permutation $(1\,5\,4\,3\,2)$.

Then, applying  \eqref{eq:univ-rep} and \eqref{Equation: equation 3},
we get
\[
(\iota\ot\iota\ot\iota\ot\iota\ot\pi)
\Bigl(\bigl((\iota\ot\chi)\circ(\iota\ot\beta^u)
  \bigl(\widehat{\wW^\qg}\bigr)\bigr)_{145}\Bigr)
  = y^\ast_{45}z_{15}y_{45},
\]
where 
$\chi : M(C^u_0(\qg)\ot C_0(\IR))\to M(C_0(\IR)\ot C^u_0(\qg))$
is the usual flip (we are also using the fact that
$\widehat{\wW^\qg} = \Sigma(\Ww^\qg)^\ast\Sigma$). 

Letting $y_t:=y(t)\in B(H)$ for $t\in\IR$,  it
follows from the above equation that
\[
\pi(\tau^u_t(x)) = y^\ast_t\pi(x)y_t,\qquad t\in\IR, x\in C^u_0(\qg).
\]
Therefore, $\pi:C^u_0(\qg)\to B(H)$ is a covariant C*-representation and
hence lifts to
a representation $\wt\pi$ of the crossed product C*-algebra
$C^u_0(\qg)\rtimes_{\beta^u}\IR$.  
Put
\[
\wW:=\ww^\IR_{24} \bigl((\iota\ot\beta^u)(\widehat{\wW^\qg})\bigr)_{134}.
\]
The above argument shows that
$X = (\id\ot\wt\pi)(\wW)$, so that $\wW$
is a maximal corepresentation of $(C_0(\qg\ltimes_\alpha\IR),\Delta)$
in the sense of \cite[Definition 23]{soltan_multiplicative_ii}.
Therefore, the right leg of $\wW$ generates the universal C*-algebra
$C^u_0(\qg\rtimes_{\beta}\IR)$, but this
is precisely $C^u_0(\qg)\rtimes_{\beta^u}\IR$, which is what we wanted
to prove. 
\end{proof}

In our case, $\alpha$ being trivial, the left Haar weight
of $\qg\ltimes_\alpha\IR$ is $\widehat\phi\ot\psi$
where $\widehat\phi$ is the left Haar weight of $\dual\qg$
and $\psi$ the Haar weight of $\IR$
(see \cite[p. 141]{double_crossed_product}). 
Hence, the modular operator of the left Haar  weight of
$\qg\ltimes_\alpha\IR$ is $\widehat{\nabla}\ot I_{L^2(\IR)}$,
where $\widehat{\nabla}$ is the modular operator of the
left Haar weight of $\widehat\phi$.
It follows that the scaling group of $\qg\rtimes_{\beta}\IR$ is given by
$\widetilde\tau_t = \tau_t\ot\iota_{B(L^2(\IR))}$, $t\in\IR$, and so
the scaling group on $C^u_0(\qg\rtimes_{\beta}\IR)$ 
is given by $\widetilde\tau^{u}_t = \tau^u_t\ot\iota_{B(L^2(\IR))}$,
$t\in\IR$.

By the general theory of crossed products, the covariant
representations of $C^u_0(\qg)$ with respect to $(\tau^u_t)$
correspond to the non-degenerate representations of 
$C_0^u(\qg)\rtimes_{\beta^u}\IR$ (note that $\IR$ is amenable
so the reduced and the full crossed products coincide).
Now suppose that $\rho$ is a non-degenerate representation of
$C_0^u(\qg)\rtimes_{\beta^u} \IR$.
Then $\pi:=\rho\circ\beta^u$ is a covariant representation
of $C_0^u(\qg)$ and the associated one-parameter unitary group
is given by $U_t = \rho(1\ot\lambda_t)$
where $\lambda_t\in M(C^*(\IR))$ is the left translation by
$t\in\IR$. 

It is easy to see that
$\wt\tau^{u}_t \circ \beta^u = \beta^u\circ \tau^u_t$ for 
all $t\in\IR$.
Now for every $f\in C_c(\IR)$ (the space of
compactly supported continuous functions on $\IR$)
and $x\in C^u_0(\qg)$, we have 
\begin{equation*}
\begin{split}
  \rho\big(\wt\tau^{u}_t\big(\beta^u(x)(1\ot f)\big)\big)
  &=\rho\big(\beta^u(\tau^u_t(x))(1\ot f)\big)
  =\pi(\tau^u_t(x))\int_\IR f(s)U_s \,ds\\
  &=U^\ast_t\Bigl(\pi(x)\int_\IR f(s)U_s \,ds\Bigr)U_t
  =U^\ast_t\bigl(\rho\big(\beta^u(x)(1\ot f)\big)\bigr)U_t.
\end{split}
\end{equation*}
The norm-density of elements of the form $\beta^u(x)(1\ot f)$
in $C^u_0(\qg)\rtimes_{\beta}\IR$ implies that 
$\rho(\wt\tau^{u}_t(X)) = U^\ast_t(\rho(X))U_t$ for all
$X\in C^u_0(\qg)\rtimes_{\beta}\IR$ and $t\in\IR$.
In particular, it follows from
Proposition \ref{Proposition: covariant finite-dimensional C*
  representations give admissible corepresentations},
that all finite-dimensional unitary representation of the
quantum group $\qg\ltimes_\alpha \IR = \widehat{\qg\rtimes_{\beta}\IR}$
are admissible. We summarise these observations in the following
theorem.

\begin{thm} \label{thm:bicrossed}
Let $\qg$ be a locally compact quantum group. Let
$(\tau_t)_{t\in\IR}$ be the scaling group and consider the matching 
\[
\tau: L^\infty(\qg)\otol L^\infty(\IR)\to L^\infty(\qg)\otol
L^\infty(\IR),
\quad \tau(f)(t) = \tau_t(f(t)) \quad(f\in L^\infty(\qg)\otol L^\infty(\IR),~t\in\IR).
\] 
Let $\alpha:L^\infty(\IR)\to L^\infty(\qg)\otol L^\infty(\IR)$
and  $\beta:L^\infty(\qg)\to
L^\infty(\qg)\otol L^\infty(\IR)$
be the associated left and right actions 
defined by the formulas
\begin{align*}
\alpha(f) &= \tau(1\ot f)=1_{L^\infty(\qg)}\ot f,\qquad f\in L^\infty(\IR),\\
\beta(x)(t) &= \tau(x\ot 1)(t) =\tau_t(x),\qquad t\in\IR,~x\in L^\infty(\qg). 
\end{align*}
Let $\qg\ltimes_{\alpha}\IR$ be the the quantum group arising from the
action $\alpha$ through the bicrossed product construction.
Then all finite-dimensional unitary representations of $\qg\ltimes_\alpha\IR$
are admissible.   
\end{thm}

\subsection{Unimodular quantum groups}
\label{sec:quantum-double}

A locally compact quantum group is \emph{unimodular}
if its left and right Haar weights coincide. 
The following theorem states that all unimodular quantum groups
satisfy the Admissibility Conjecture.

\begin{thm}\label{thm:admissibility-unimodular}
  Suppose that $\qg$ is a unimodular locally compact quantum group.
  Then all finite-dimensional unitary representations of $\qg$ are admissible.  
\end{thm}

\begin{proof}
Let $U\in M(C_0(\qg)\ot M_n(\IC))$ be a finite-dimensional unitary
representation, and let $V\in M(C^u_0(\qg)\ot M_n(\IC))$ be the
unique lift of $U$ to the universal level.
Write $V =(V_{ij})$, where $V_{ij}$ are the matrix coefficients of
$V$, and let
\[
L_V = \linsp\set{V_{ij}}{i, j=1, 2, \ldots, n}.
\]
Since $\qg$ is unimodular, the universal modular automorphism groups
of the right and left invariant weights are the same. It then follows from 
\cite[Proposition~9.2]{kus} that
\[
\Delta_u\circ\sigma^u_t = (\tau^u_t\ot\sigma^u_t)\circ\Delta_u
\]
and
\[
\Delta_u\circ\sigma^u_t=(\sigma^u_t\ot\tau^u_{-t})\circ\Delta_u
\]
for every $t\in\IR$, where $\{\sigma^u_t\}_{t\IR}$ denotes the universal modular automorphism group of both the right and left invariant weights. Applying these to the matrix coefficient of $V$,
we obtain 
\[
\Delta_u(\sigma^u_t(V_{ij}))=\sum_{k=1}^n\tau_t^u(V_{ik})\ot\sigma^u_t(V_{kj})
\]
and
\[
\Delta_u(\sigma^u_t(V_{ij}))=\sum_{k=1}^n\sigma^u_t(V_{ik})\ot\tau^u_{-t}(V_{kj}). 
\]
Then applying the counit of $C_0^u(\qg)$ to the above
identities, it follows that $\tau^u_{-t}(\sigma^u_{t}(L_V))\subset L_V$ and
$\tau^u_t(\sigma^u_t(L_V))\subset L_V$ for all $t\in\IR$.
This implies that for $X\in L_V$, 
\[
\tau^u_{2t}(X) = \big(\tau^u_{t}\circ\sigma^u_{-t}\big)\circ
\big(\sigma^u_{t}\circ\tau^u_t\big)(X)\subset L_V.
\]
Therefore, $\tau^u_t(L_V)\subset L_V$ for all $t\in\IR$.

Let $\Lambda_\qg:C^u_0(\qg)\to C_0(\qg)$ be the reducing
morphism. Then $\tau_t\circ\Lambda_\qg=\Lambda_\qg\circ\tau^u_t$ for all
$t\in\IR$ (see \cite[Section 4]{kus}).
Write $U = (U_{ij})$, and note that $\Lambda_\qg(V_{ij})=U_{ij}$.
Since $L_V$ is $\tau^u_t$-invariant, it follows that
$L_U$ is $\tau_t$-invariant. By Proposition \ref{Proposition:
  covariant finite-dimensional C* representations give admissible
  corepresentations}, $U$ is admissible.
\end{proof}

A class of  examples of unimodular quantum groups with non-trivial scaling
groups is given by Drinfeld doubles. 
Given a matched pair of locally compact quantum groups $\qg$ and
$\IH$, the von Neumann algebra $L^\infty(\qg)\otol L^\infty(\IH)$ can
be given the structure of a locally compact quantum group, called the
double crossed product of $\qg$ and $\IH$
\cite[Section 3 \& Theorem 5.3]{double_crossed_product}.
If $\IH=\qgdual$, then a matching
\[
m:L^\infty(\qg)\otol L^\infty(\qgdual)\to L^\infty(\qg)\otol L^\infty(\qgdual)
\]
is given by
\[
m(X)=W_{\qg^\mathrm{op}} X W_{\qg^\mathrm{op}}^\ast,
\]
where $W_{\qg^\mathrm{op}}$ is the left multiplicative unitary of the
opposite quantum group $\qg^\mathrm{op}$ obtained from $\qg$.
The resulting double crossed product quantum group is precisely 
Drinfeld double, as has been shown in
\cite[Section~8]{double_crossed_product}. Moreover,
by \cite[Proposition 8.1]{double_crossed_product}, 
Drinfeld doubles are always unimodular.
It also follows from \cite[Theorem 5.3]{double_crossed_product} that
if $\qg$ has non-trivial scaling group, then the double crossed product has
non-trivial scaling group as well. Thus Drinfeld double
construction produces concrete examples of unimodular locally compact
quantum groups with non-trivial scaling group, and by
Theorem \ref{thm:admissibility-unimodular} 
the Admissibility Conjecture is true for such quantum groups.

\section{Characterisation of Kazhdan's Property T for locally compact
  quantum groups}
\label{Section: Fell topology and weak containment of representations
  of C*-algebras} 

Let $A$ be a C*-algebra and let $\pi:A\to B(H_\pi)$ and
$\rho:A\to B(H_\rho)$ be two non-degenerate
representations. The representation
$\pi$ is said to be \emph{equivalent} to $\rho$ if there
exists a unitary $U:H_\pi\to H_\rho$ such that
$U^\ast\rho(x)U = \pi(x)$ for every $x\in A$. If $U$ is only an
isometry, then we will say that $\pi$ is \emph{contained} in $\rho$
and write $\pi\subset\rho$
(in other words, $\pi$ is a \emph{subrepresentation} of $\rho$).
Now let $\cls$ be a set of representations of $A$. 
We say that the representation $\pi$ is
\emph{weakly contained} in $\cls$
and write $\pi\prec\cls$ if
$\bigcap_{\rho\in\cls}\Ker\rho\subset \Ker\pi$. 
We will adopt the convention that whenever $\pi\prec\{\rho\}$ we will
simply write $\pi\prec\rho$. 

Let $\widehat{A}$ denote the set of inequivalent irreducible
representations of $A$. The \emph{closure} of $\cls\subset\widehat{A}$ is
defined as 
\[
\overline{\cls}= \set{\pi\in\widehat{A}}{\pi\prec\cls}. 
\]
From \cite[Lemma 1.6]{fell} it follows that the above closure defines 
a topology on $\widehat{A}$, which is referred to as the Fell topology
on $\widehat{A}$ (in \cite{fell} this was called the hull-kernel topology on
$\widehat{A}$). The following result
from \cite[Lemma~2.1]{xiao_property_T} is crucial in the sequel.

\begin{ppsn}\label{prop:T} 
Let $A$ be a C*-algebra. 
If $\rho\in\widehat{A}$
is finite-dimensional, then $\{\rho\}$ is a closed subset of
$\widehat{A}$. 
Moreover, the following statements are equivalent for $\rho\in\widehat{A}$:
\begin{enumerate}[label=\textup{(\roman*)}]
\item $\rho$ is an isolated point in $\widehat{A}$.
\item If a representation $\pi$ of $A$ satisfies $\rho\prec\pi$, then
  $\rho\subset\pi$. 
\item $A=\Ker\rho\oplus\bigcap_{\nu\in\widehat{A}\setminus\{\rho\}}\Ker\nu$.
\end{enumerate}
\end{ppsn}

\begin{rmrk}\label{Remark: Notion of weak containment from Brannan}
  It is worthwhile to mention that recently in
  \cite[Definition~3.3]{brannan_property_T} the authors have used a
  slightly different notion of weak containment, namely a C*-algebraic
  version of \cite[Definition~7.3.5]{Zimmer_weak_containment}.
  However, we will be concerned with irreducible representations
  in which case the definition coincide
  (as mentioned after Definition 3.3 in \cite{brannan_property_T}).
\end{rmrk}

The following definition is a natural extension of containment
to the setting of unitary representations of locally compact quantum
groups (see \cite[Definition~3.2]{haagerup_daws_fima_skalski}). 

\begin{dfn}
Let $U\in M(C_0(\qg)\ot \K(H))$ and $V\in M(C_0(\qg)\ot \K(K))$ be two
unitary representations of a locally compact quantum group $\qg$,
and let $\pi_U:C^u_0(\qgdual)\to B(H)$ and
$\pi_V:C^u_0(\qgdual)\to B(K)$ be the associated 
representations  of $C^u_0(\qgdual)$
(i.e.  $U=(\iota\ot\pi_U)(\wW)$ and $V=(\iota\ot\pi_V)(\wW)$).
If $\pi_U\prec\pi_V$, we say $U$ is \emph{weakly contained} in $V$
and write $U\prec V$.  If $\pi_U\subset\pi_V$, we 
say that $U$ is \emph{contained} in $V$
(or that $U$ is a \emph{subrepresentation} of $V$) and write $U\subset V$. 
\end{dfn}

\begin{rmrk}\label{Remark: Tensor Containment}
Let $U \prec V$, so that $\pi_U \prec \pi_V$.  Let $W$ be a corepresentation.
We claim that $W \tp U \prec W \tp V$.  Indeed, this is equivalent to
$\pi_{W\tp U} = (\pi_W\ot\pi_U)\circ\chi\circ\widehat\cop_u
\prec (\pi_W\ot\pi_V)\circ\chi\circ\widehat\cop_u = \pi_{W\tp V}$, which
by definition, is equivalent to showing that $\ker (\pi_W\ot\pi_V)\circ\widehat\cop_u
\subseteq \ker (\pi_W\ot\pi_U)\circ\widehat\cop_u$.

Let $W$ act on $H_W$, and let $a\in \ker (\pi_W\ot\pi_V)\circ\widehat\cop_u$.  For
$\omega \in \K(H_W)^*$ let $b = (\omega\circ\pi_W\ot\id)\widehat\cop_u(a)$, so
that $\pi_V(b) = 0$.  As $\pi_U\prec\pi_V$, it follows that $\pi_U(b)=0$, so
that $(\omega\ot\id) (\pi_W\ot\pi_U)\circ\widehat\cop_u(a)=0$.  As $\omega$ was
arbitrary, $(\pi_W\ot\pi_U)\circ\widehat\cop_u(a)=0$, as required.
\end{rmrk}

We next recall the definitions of invariant and almost invariant
vectors for quantum group representations
(see \cite[Definition 3.2]{haagerup_daws_fima_skalski}).

\begin{dfn}\label{Definition: Invariant and almost invariant vectors}
Let $U\in M(C_0(\qg)\ot \K(H))$ be a representation of a locally
compact quantum group $\qg$. A vector $\xi\in H$ is called
\emph{invariant} for $U$ if $U(\eta\ot\xi)=\eta\ot\xi$ for all
$\eta\in L^2(\qg)$.   
$U$ is said to have \emph{almost invariant vectors} if there exits a net
$(\xi_\alpha)_\alpha$ of unit vectors in $H$ such that
$\|U(\eta\ot\xi_\alpha)-\eta\ot\xi_\alpha\|\to0$ for all
$\eta\in L^2(\qg)$. 
\end{dfn}

Note that if $U$ and $V$ are similar representations of a locally
compact quantum group $\qg$, then
$U$ has an invariant vector if and only if $V$ has.
The analogous statement holds for almost invariant vectors.

We will need invariant means in the following arguments,
so next we set up the necessary terminology for those.

\begin{dfn}\label{Definition: Eberlein algebra for a LCQG}
The (reduced) \emph{Fourier--Stieltjes algebra} of $\qg$ is defined by  
\[
\fs(\qg) = \linsp\set{(\iota\ot\omega)(\wW)}{\omega\in C^u_0(\qgdual)^\ast}. 
\]
Then the \emph{Eberlein algebra} of $\qg$ is defined by   
\[
\eb(\qg) = \overline{\fs(\qg)}^{\|\cdot\|_{B(L^2(\qg))}}. 
\]
\end{dfn}

Note that $\fs(\qg)\subset M(C_0(\qg))$ and that
$\fs(\qg)$ is a subalgebra of $M(C_0(\qg))$.
It then follows that $\eb(\qg)$ is a closed subalgebra
of $M(C_0(\qg))$. When $\qg$ is of Kac type, $\eb(\qg)$ is
self-adjoint and so a C*-subalgebra of $M(C_0(\qg))$
(see for example \cite[Section 7]{matbd}). 

The Banach algebra $L^1(\qg)$ acts on its dual $L^\infty(\qg)$  by
\[
x\cdot\omega =(\omega\ot\iota)(\Delta(x)),
\qquad \omega\in \lone(\qg), x\in\linf(\qg), 
\]
and
\[
\omega\cdot x=(\iota\ot\omega)(\Delta(x)), \qquad \omega\in
\lone(\qg), x\in\linf(\qg),
\]
making $\linf(\qg)$ an $\lone(\qg)$-bimodule. It is easy to see that
$\eb(\qg)$ is invariant under these actions. This allows us to define
the notion of an invariant mean on $\eb(\qg)$.  It follows from
\cite[Proposition 3.15]{daws_skalski_viselter}
that $\eb(\qg)$ admits an \emph{invariant mean} $\mu\in \eb(\qg)^*$
in the sense that $\|\mu\| = \mu(1) = 1$ and 
\[
\mu(\omega\cdot x) = \omega(1) \mu(x) = \mu(x\cdot\omega)
\qquad \omega\in L^1(\qg), x\in \eb(\qg).
\]

Note that the Hopf $*$-algebra $\aphopf(\qg)$ underlying the quantum
Bohr compactification is contained in $\fs(\qg)$, and therefore
$\ap(\qg)\subset \eb(\qg)$. The uniqueness of the Haar state of a
compact quantum group implies the following result. 

\begin{lmma}\label{Lemma: the restriction of the invariant mean to the Bohr compactification is the Haar state}
  Let $M$ be the restriction of the invariant mean $\mu\in \eb(\qg)^*$
  to $\ap(\qg)$. Then $M$ is the Haar state of the
  compact quantum group $(\ap(\qg),\Delta)$.  
\end{lmma}

\begin{lmma}\label{Lemma: making sense of tensoring with the mean}
Let $X\in B(L^2(\qg))\otol B(H)$ for some Hilbert space $H$, and suppose
that $(\iota\ot\omega)(X)\in \eb(\qg)$ for all $\omega\in B(H)_\ast$.
Let $\nu\in \eb(\qg)^*$. Then there exists an operator
$T\in B(H)$ such that for all $\omega\in B(H)_\ast$  
\[
\langle T, \omega\rangle = \langle \nu, (\iota\ot\omega)(X)\rangle.
\]
We will be denoting this operator by $(\nu\ot\iota)(X)$, so that
\[
\omega\bigl((\nu\ot\iota)(X)\bigr) = \nu\bigl((\iota\ot\omega)(X)\bigr)
\]
for all $\omega\in B(H)_\ast$.
\end{lmma}

\begin{proof}
The map
\[
\omega\mapsto\langle \nu, (\iota\ot\omega)(X)\rangle
: B(H)_\ast \to \IC
\]
defines a bounded functional on $B(H)_\ast$, which gives the existence
of the operator $T\in B(H)$. 
\end{proof}
The next result gives a formula for the orthogonal projection onto the set of invariant vectors for a unitary corepresentation of $C_0(\qg)$ (also see \cite[Proposition 3.14]{daws_skalski_viselter}).
\begin{lmma}\label{Corollary: expressing the projection onto the set of invarint subspace in terms of the invariant mean}
  Let $V\in M(C_0(\qg)\ot \K(H))$ be a unitary representation of $\qg$
  and let $\mu\in \eb(\qg)^*$ be the invariant mean on $\eb(\qg)$.
  The operator $P = (\mu\ot\iota)(V)\in B(H)$
  is the projection onto the subspace of invariant vectors for $V$ (we may note that the definition of $P$ makes sense because of Lemma \ref{Lemma: making sense of tensoring with the mean}).
\end{lmma}

\begin{proof}
  Due to the subtle definition of $(\mu\ot\iota)(V)$, we include
  a careful calculation of the fact that the image of $P$
  consists of invariant vectors. Given $\sigma\in \lone(\qg)$
  and $\omega\in B(H)_*$, we have
  \[
  (\sigma\ot\omega)\bigl(V(1\ot P)\bigr)
  = \omega((\sigma\ot \id)(V)P)
   = \mu\bigl((\sigma\ot\id\ot\omega)(V_{13}V_{23})\bigr)
   \]
   due to the commutation relation in
   Lemma~\ref{Lemma: making sense of tensoring with the mean}.
   Continuing from here using the fact that $V$ is a representation,
   we have
  \begin{align*} 
   (\sigma\ot\omega)\bigl(V(1\ot P)\bigr)
   &= \mu\bigl((\sigma\ot\id\ot\omega)((\cop\ot\id)(V))\bigr)
   = \mu\bigl((\sigma\ot\id)(\cop((\id\ot\omega)(V)))\bigr)\\
   &= \sigma(1)\mu((\id\ot\omega)(V))
   = (\sigma\ot\omega)(1\ot P)
   \end{align*}
   due to the invariance of $\mu$ and the fact that $\mu((\id\ot\omega)(V))=\omega(P)$ by virtue of Lemma \ref{Lemma: making sense of tensoring with the mean}. 
   It follows that the image of $P$ consists of invariant vectors.
   That $P$ is an idempotent map is also easy to compute.  
   Notice that $\|P\|\leq 1$, because $\|\mu\|=1$, from which it follows that
   $P$ is a subprojection of the orthogonal projection onto the subspace of invariant vectors.
   
   We now prove that indeed $P$ is the projection onto the subspace of invariant vectors. We show that for any invariant vector $\xi$, $P\xi=\xi$, from which the result follows. So let then $\xi$ be an invariant vector. Let $\eta\in H$. Letting $\omega_{\xi,\eta}(\cdot)\in B(H)_\ast$ be the functional given by $\omega_{\xi,\eta}(x):=\langle x\xi,\eta\rangle$ for $x\in B(H)$  we have
   \begin{align*}
   \omega_{\xi,\eta}(P)& = \omega_{\xi,\eta}\big((\mu\ot\iota)(V)\big)\\
   &=\mu\big((\iota\ot\omega_{\xi,\eta})(V)\big)\quad(\text{by Lemma \ref{Lemma: making sense of tensoring with the mean}})\\
   &=\mu(1.\omega_{\xi,\eta}(1))\quad(\text{as $\xi$ is invariant})\\
   &=\omega_{\xi,\eta}(1)\quad(\text{as $\mu(1)=1$}).
   \end{align*}
   As this holds for any $\eta\in H$, we have that $P\xi=\xi$, as we wanted.
\end{proof}

The following result from
\cite[Corollary~2.5 and Corollary~2.8]{haagerup_daws_fima_skalski}
gives a nice criterion for the existence of invariant and almost
invariant vectors. Denote the trivial representation of $\qg$ by~$1$.

\begin{ppsn}
\label{prop:inv-cont} 
Let $U\in M(C_0(\qg)\ot \K(H))$ be a unitary
representation of a locally compact quantum group $\qg$.  
\begin{enumerate}[label=\textup{(\roman*)}]
 \item 
 $U$ has a nonzero invariant vector if and only if $1\subset U$.
 \item
 $U$ has almost invariant vectors if and only if $1\prec U$.
\end{enumerate}
\end{ppsn}

We now recall the definition of Kazhdan's Property (T) for quantum groups
(see \cite[Definition 3.1]{xiao_property_T}, which
goes back to \cite[Definition 6]{Fima_property_T_discrete_quantum_group}).

\begin{dfn}\label{Definition: Kazhdan property (T) for QG}
  A locally compact quantum group $\qg$ has (\emph{Kazhdan's})
  \emph{Property (T)}
  if every unitary representation of $\qg$
  that has almost invariant vectors has a nonzero invariant vector.
\end{dfn}

It follows from Proposition~\ref{prop:inv-cont}
that $\qg$ has Property (T) if and only if
for every unitary representation $U$ of $\qg$
\[
1\prec U\implies 1\subset U.
\]

The following theorem
\cite[Proposition 3.2 and Theorem 3.6]{xiao_property_T}
and \cite[Theorem 4.7]{brannan_property_T} 
gives a series of equivalent conditions to Property (T). 

\begin{thm}\label{thm:T-equi}
  Let $\qg$ be a locally compact quantum group.
The following statements are equivalent:
\begin{enumerate}[label=\textup{(T\arabic*)}]
\item $\qg$ has Property (T).
\item The counit $\widehat{\epsilon}_u$ is an isolated point in
  $\widehat{C^u_0(\qgdual)}$.
\item $C^u_0(\qgdual)=\Ker\widehat{\epsilon}_u\oplus\IC$.
\item There is a projection $p\in M(C^u_0(\qgdual))$ such that
  $p C^u_0(\qgdual) p = \IC p$ and $\widehat{\epsilon}_u(p)=1$. 
\end{enumerate}
If $\qg$ has trivial scaling group, then the above conditions are further
equivalent to the following (quantum version of Wang's theorem \cite[Theorem 2.1]{wang_property_T}):
\begin{enumerate}[label=\textup{(T\arabic*)}] \setcounter{enumi}{4}
 \item 
Every finite-dimensional irreducible representation of
$C^u_0(\qgdual)$ is an isolated point in $\widehat{C^u_0(\qgdual)}$. 
\item
  $C^u_0(\qgdual)\cong B\oplus M_n(\IC)$ for some C*-algebra $B$ and
  $n\in\IN$. 
\end{enumerate}
\end{thm}

We will prove that for a locally compact quantum group with
non-trivial scaling group, (T5) as well as 
a suitable generalisation of (T6) are equivalent with (T1)--(T4). 

We first make two observations, which will be used later (also see \cite[Proposition 3.14~\&~Proposition 3.15]{viselter}). 

\begin{lmma}\label{lemma:trivial}
Let $U\in M(C_0(\qg)\ot M_n(\IC))$ be a
finite-dimensional admissible unitary representation of $\qg$. Then
$1\subset U^c\tp U$. 
\end{lmma}

\begin{proof}
  Write $U=(U_{ij})_{i,j=1}^n$, and define $\overline{U} = (U^*_{ij})_{i,j=1}^n$. 
  Since $U$ is admissible, $\overline{U}\in \eb(\qg)\ot M_n(\IC)$.
  Define $V =\overline{U}\tp U \in \eb(\qg)\ot (M_n(\IC)\ot M_n(\IC))$.
  Let $\mu$ be the invariant mean on  $\eb(\qg)$.
  We will show that $(\mu\ot\iota)(V)\neq0$, where we are using the
  notation of Lemma \ref{Lemma: making sense of tensoring with the
    mean}. 
Towards a contradiction, suppose that $(\mu\ot\iota)(V)=0$, so that for all
$\nu\in (M_n(\IC)\ot M_n(\IC))_\ast$ we have
$\mu((\id\ot \nu)(V))=0$. Choosing $\nu$ such that
$(\iota\ot \nu)(V)=U^\ast_{ij}U_{ij}$
leads to $\mu(U^\ast_{ij}U_{ij})=0$ for all $i,j = 1$,~$2$,
\ldots,~$n$. Therefore, 
\[
\mu\biggl(\sum_{i=1}^nU^\ast_{ij}U_{ij}\biggr) = \mu(1) = 0,
\]
which cannot happen as $\mu(1) = 1$. 

By Lemma~\ref{Corollary: expressing the projection onto the set of invarint subspace in terms of the invariant mean},
any vector in the image of $(\mu\ot\iota)(V)$
is an invariant vector for $V$.
Since $\overline{U}$ is similar to $U^c$ (as $U$ is admissible, Remark~\ref{Remark: CQG}), the
representation $V$ is similar to $U^c\tp U$ and  the result follows from
Proposition~\ref{prop:inv-cont}-(i). 
\end{proof}

The following result may be considered as an extension
of \cite[Proposition 3.5]{xiao_property_T}
(see also \cite[Proposition A.1.12]{bekka_property_T} and
\cite[Theorem 2.6]{kyed}).  

\begin{thm}
\label{Theorem: Tensor product contains trivial representation implies
  that each contains a finite-dimensional representation} 
Let $U\in M(C_0(\qg)\ot M_n(\IC))$ and $V\in M(C_0(\qg)\ot \K(H))$ be
unitary representations of a locally compact quantum group $\qg$ with
$U$ being admissible. 
Then the following are equivalent:
\begin{enumerate}[label=\textup{(\roman*)}]
 \item The representations $U$ and $V$ contain a common
   finite-dimensional unitary representation.  
\item The representation $U^c\tp V$ contains the trivial representation.
\end{enumerate}
\end{thm}

\begin{proof}
(i)$\implies$(ii): 
Let $W$ be a common finite-dimensional unitary representation of $U$
and~$V$. Then $W$ is also admissible as $U$ is admissible. We have
$W^c\tp W\subset U^c\tp V$ and by Lemma~\ref{lemma:trivial},
$1\subset W^c\tp W$.

(ii)$\implies$(i):
Let $\mu$ denote the invariant mean on $\eb(\qg)$.
Let $x\in B(H,\IC^n)$  and 
$y=(\mu\ot\iota)(U^*(1\ot x)V)$ so that $y\in B(H,\IC^n)$. 
Now
\begin{align*}
  U^*(1\ot y)V
  &= (\mu\ot\id\ot\id)\bigl(U^*_{23}U^*_{13}(1\ot1\ot x)V_{13}V_{23}\bigr)\\
 &= (\mu\ot\id\ot\id)\circ(\cop\ot \id)\bigl(U^*(1\ot x)V\bigr)
=  1\ot y
\end{align*}
due to the invariance of $\mu$ (see the proof
of Lemma~\ref{Corollary: expressing the projection onto the set of
  invarint subspace in terms of the invariant mean}
for making the above calculation more rigorous).
Since $U$ is unitary,  we have
\[
U(1\ot y)=(1\ot y)V. 
\]
Next we show that for some $x$, the resulting $y$ is nonzero, 

It follows from the hypothesis, via Proposition~\ref{prop:inv-cont}-(i), 
that $\conj{U}\tp V$ has an invariant vector $\zeta\in \IC^n\ot H$,
and so
\begin{equation} \label{eq:zeta-eq}
(\id\ot \omega_{\zeta,\zeta})(\overline{U}\tp V) = \la \zeta, \zeta\ra 1.
\end{equation}
For $\xi\in H$ and $\alpha\in \IC^n$,
let $x = \theta_{\alpha,\xi}\in B(H, \IC^n)$ be defined by 
$\theta_{\alpha,\xi}(\eta)=\la \eta,\xi\ra\alpha$ for $\eta\in H$.
If $y=(\mu\ot\iota)(U^*(1\ot\theta_{\alpha,\xi})V) = 0$
for every $\xi\in H$ and $\alpha\in \IC^n$,
then for every $\alpha,\beta\in\IC^n$ and $\xi,\eta\in H$, we have  
\begin{equation*}
\begin{split}
0=\la y\eta,\beta\ra
&= \mu\bigl((\iota\ot\omega_{\eta,\beta})(U^\ast(1\ot\theta_{\alpha,\xi})V)\bigr)
=\mu\bigl(
  (\iota\ot\omega_{\beta,\alpha})(U)^\ast(\iota\ot\omega_{\eta,\xi})(V)\bigr)\\
&=\mu\bigl( (\iota\ot\omega_{\alpha,\beta})(\overline{U})
           (\iota\ot\omega_{\eta,\xi})(V)\bigr).
\end{split}
\end{equation*}
Therefore $(\mu\ot\iota)(\overline{U}\tp V) = 0$, and so
by equation \eqref{eq:zeta-eq}
\[
0 = \omega_{\zeta,\zeta}\bigl((\mu\ot\iota)(\overline{U}\tp V)\bigr)
= \mu\bigl((\id\ot \omega_{\zeta,\zeta})(\overline{U}\tp V)\bigr)
= \la\zeta, \zeta\ra \mu(1) \ne 0.
\]
Consequently, $y\ne 0$ for some $\xi\in H$, $\alpha\in \IC^n$.

To finish the proof we now argue as in the last part of the proof of
\cite[Theorem 2.6]{kyed}. Since $U$ and $V$ are unitaries such that
$U(1\ot y)=(1\ot y)V$, it follows that
$V(1\ot y^\ast)=(1\ot y^\ast)U$. Moreover, since 
$y$ is a compact operator, this implies that also $y^\ast y$ is
compact and satisfies $V(1\ot y^\ast y)=(1\ot y^\ast y)V$. Similarly we have
that $U(1\ot yy^\ast)=(1\ot yy^\ast)U$, where $yy^\ast\in
B(\IC^n)$. Both $y^\ast y$ and $yy^\ast$ are positive compact
operators, which implies that the eigenspaces of positive eigenvalues
are finite-dimensional. Moreover, the non-zero part of the spectrum of
$y^\ast y$ and $yy^\ast$ are the same. So let $\lambda>0$ be a common
eigenvalue for both $y^\ast y$ and $yy^\ast$. Then the eigenprojection
of $yy^\ast$ corresponding to $\lambda$ also intertwines $U$, thereby
producing a finite-dimensional subrepresentation of $U$, say
$W$. Applying the same argument to $y^\ast y$, we get a
finite-dimensional subrepresentation of $V$, say $W^\prime$. Since $U$ is
admissible by hypothesis, also  $W$ is admissible.
Moreover, the partial isometry coming from the polar
decomposition of $y$ gives an equivalence between $W$ and $W^\prime$,
which proves the claim.
\end{proof}

We now prove a generalisation of conditions (T5) and (T6) in
Theorem~\ref{thm:T-equi}. 

\begin{thm} \label{thm:kazhdan}
 Let $\qg$ be any locally compact quantum group. Then 
 $\qg$ having Property (T) is equivalent to any of the following
 statements:
\begin{enumerate}[label=\textup{(T\arabic*)}] \setcounter{enumi}{4}
 \item 
Every finite-dimensional irreducible C*-representation of
$C_0^u(\qgdual)$ is an isolated point in $\widehat{C^u_0(\qgdual)}$.  
\item[\textup{(T6$'$)}] 
There is a finite-dimensional irreducible C*-representation of
$C_0^u(\qgdual)$ which is covariant with respect to the scaling
automorphism group $(\widehat{\tau}^u_t)$
and is an isolated point in $\widehat{C^u_0(\qgdual)}$.  
\item
$C^u_0(\qgdual)\cong B\oplus M_n(\IC)$ for some C*-algebra $B$ and
  some $n\in\IN$ with $\widehat{\tau}^u_t(B)\subset B$ for all $t\in\IR$. 
\end{enumerate}
\end{thm}

\begin{proof}
  The proof is based on the same idea as the proof of
  \cite[Theorem~3.6]{xiao_property_T}.  

  (T5)$\implies$(T6$'$) because covariant irreducible finite-dimensional
  representations always exist: the counit.
  
  (T6$'$)$\implies$(T6):
  Let $\pi$ be a finite-dimensional
  irreducible representation of $C^u_0(\qgdual)$
  which is covariant with respect to the scaling automorphism group.  
  By the equivalence of (i) and (iii) in Proposition~\ref{prop:T},
  $C^u_0(\qgdual)\cong \ker \pi \oplus M_n(\IC)$ for some $n\in\IN$.
  Since $\pi$ is covariant, we have
  $\widehat{\tau}^u_t(\ker(\pi))\subset \ker\pi$ for all $t\in\IR$,
  and so (T6) holds.

  It remains to prove that (T1) of
  Theorem~\ref{thm:T-equi} implies (T5) and that (T6) implies (T2) of
  Theorem~\ref{thm:T-equi}. The result will then follow from the 
  equivalence of (T1) and (T2) of Theorem~\ref{thm:T-equi}. 

(T1)$\implies$(T5):
To prove (T5), it is enough to show that (ii) in Proposition~\ref{prop:T} holds
for every irreducible finite-dimensional representation $\rho$ of
$C^u_0(\qgdual)$. To this end, let $\pi$ be a representation of
$C_0^u(\qgdual)$ such that $\rho\prec\pi$.
Let $U=(\iota\ot\rho)(\wW)$ and $V=(\iota\ot\pi)(\wW)$.
Since $\qg$ has Property (T), it is unimodular by
Theorem~\ref{thm: property T implies unimodularity}, and so
$\qg$ satisfies the Admissibility Conjecture
by Theorem~\ref{thm:admissibility-unimodular}.
Hence $U$ is an irreducible finite-dimensional
admissible unitary representation of $\qg$
and $U\prec V$. Now by Lemma 
\ref{lemma:trivial} we have that $1\subset U^c\tp U$ and also
$U^c\tp U\prec U^c\tp V$ by Remark~\ref{Remark: Tensor Containment}.
Since $\qg$ has Property (T) and $1\prec U^c\tp V$, it follows
that $1 \subset U^c \tp V$. Thus by Theorem \ref{Theorem: Tensor
  product contains trivial representation implies that each contains a
  finite-dimensional representation} there exists a finite-dimensional
unitary representation $W$ such that $W\subset U$ and
$W\subset V$. Since $U$ is irreducible, this implies that $W=U$, and so
$\rho\subset\pi$.

(T6)$\implies$(T2):
Suppose that $C^u_0(\qgdual)\cong B\oplus M_n(\IC)$ and $B$ in
invariant under the scaling automorphism group.
Let $\mu:C^u_0(\qgdual)\to M_n(\IC)$ be the irreducible representation
corresponding to the summand $M_n(\IC)$, and let 
$U=(\iota\ot\mu)(\wW)$ be the unitary representation of $\qg$
associated with $\mu$. Since $\ker(\mu)\cong B$ and
$\widehat{\tau}^u_t(B)\subset B$ for all $t\in\IR$, it follows that
$\mu$ is covariant with respect to the scaling automorphism group of
$C^u_0(\qgdual)$. Hence $U$ is admissible by Proposition
\ref{Proposition: covariant finite-dimensional C* representations give
  admissible corepresentations}. 
The representation $\pi$ of $C^u_0(\qgdual)$
associated to $U^c\tp U$ is finite-dimensional and covariant with
respect to scaling automorphism of $C^u_0(\qgdual)$ by Proposition
\ref{Proposition: covariant finite-dimensional C* representations give
  admissible corepresentations} since $U^c\tp U$ is admissible.
Therefore, $\pi$ decomposes into a direct sum of finitely many
covariant irreducible representations. Since
$1\subset U^c\tp U$ by Lemma~\ref{lemma:trivial},
one of the irreducible components is $\widehat{\epsilon}_u$.
Therefore $\pi=\bigoplus_{k=0}^m\pi_k$, where $\pi_0 =
\widehat{\epsilon}_u$ and $\pi_k$ is
an irreducible finite-dimensional covariant representation
for each $k=1,2,\ldots, m$. By Proposition~\ref{prop:T},
each singleton set $\{\pi_k\}$ is closed in
the Fell topology. Towards a contradiction, let us assume that
$\widehat{\epsilon}_u$ is not an isolated point.
So there is a net
\[
(\rho_\alpha)_{\alpha\in\Lambda} \subset
\widehat{C_0^u(\qgdual)}\setminus
\{\widehat{\epsilon}_u,\pi_1,\pi_2,\ldots,\pi_m\} 
\]
such that $(\rho_\alpha)_{\alpha\in\Lambda}$ converges in the Fell topology to
$\widehat{\epsilon}_u$. By the definition of closure in the
Fell topology, this implies that
$\widehat{\epsilon}_u \prec \bigoplus_{\alpha\in\Lambda}\rho_\alpha$,
and so 
\[
\mu=(\mu\ot\widehat{\epsilon}_u)\circ\chi\circ\widehat{\cop}_u
\prec\bigoplus_{\alpha\in\Lambda}(\mu\ot\rho_\alpha)\circ\chi\circ\widehat{\cop}_u,
\]
where $\chi$ is the flip map.
Since $\mu$ satisfies condition (iii) in Proposition~\ref{prop:T}
(by definition), it follows by condition (ii) in
Proposition~\ref{prop:T} that  
\[ 
\mu\subset\bigoplus_{\alpha\in\Lambda}
(\mu\ot\rho_\alpha)\circ\chi\circ\widehat{\cop}_u. 
\]
By irreducibility of $\mu$ we have
$\mu\subset(\mu\ot\rho_{\alpha})\circ\chi\circ\widehat{\cop}_u$ for some
$\alpha\in\Lambda$. Letting $U_{\alpha}=(\iota\ot\rho_{\alpha})(\wW)$, this
means that $U\subset U\tp U_{\alpha}$. Combining this with Lemma
\ref{lemma:trivial} it follows that  
\[
1\subset U^c\tp U\subset U^c\tp U\tp U_{\alpha},  
\]
so that we have 
\[
1\subset U^c\tp U\tp U_{\alpha}
=\Big(\bigoplus_{k=0}^m (\iota\ot\pi_k)(\wW)\Big)\tp U_{\alpha}.  
\]
Since $(U^c\tp U)^c$ is equivalent to $U^c\tp U$ (recall that $U$ is
admissible, Remark~\ref{Remark: CQG}), it follows that   
\[
1\subset\Big(\bigoplus_{k=0}^m(\iota\ot\pi_k)(\wW)\Big)^c
\tp U_{\alpha}. 
\]
An application of Theorem \ref{Theorem: Tensor product contains
  trivial representation implies that each contains a
  finite-dimensional representation} now yields a finite-dimensional
unitary representation $W$ such that
$W\subset\bigoplus_{k=0}^m(\iota\ot\pi_k)(\wW)$ and
$W\subset U_{\alpha}$.
Since $U_{\alpha}$ is irreducible, we have $W=U_{\alpha}$.
On the other hand, $\pi_k$ is irreducible 
for $k=0,1,2,\ldots, m$, and so
$U_{\alpha}=(\iota\ot\pi_{k_0})(\wW)$ for some
$k_0\in\{0,1,2,\ldots, m\}$. This means that $\rho_{\alpha}=\pi_{k_0}$
which is a contradiction. Thus $\widehat{\epsilon}_u$ must be
an isolated point. 
\end{proof}

\section{Properties of quantum groups with Property (T)}

In this section we prove several properties shared by quantum groups
with Property (T). We include a very short proof of the fact that a
quantum group with Property (T) is unimodular (see \cite[Section
  6]{brannan_property_T} and \cite[Proposition
  7]{Fima_property_T_discrete_quantum_group}). We consider unimodular
locally compact quantum groups as well as quantum groups arising
through the bicrossed product construction as discussed in Subsection
\ref{Subsection: Examples of locally compact  quantum groups with
  admissible representations} and prove a variation of Theorem
\ref{thm:kazhdan} and improved versions of the quantum versions of
the Bekka--Valette theorem \cite[Theorem 4.8]{brannan_property_T}
(characterising Property (T) in terms of non-existence of almost
invariant vectors for weakly mixing representations) and the Kerr--Pichot
theorem \cite[Theorem 4.9]{brannan_property_T} (characterising
Property (T) in terms of density properties of weakly mixing
representations) for these quantum groups. 

\subsection{Quantum groups with Property (T) are unimodular}
It is a well-known fact that a locally compact group $G$ with Property
(T) is unimodular \cite[Corollary 1.3.6-(ii)]{bekka_property_T}. The
proof of this result makes use of the fact that if $G$ has Property
(T) and admits a continuous homomorphism into a locally compact
group $H$ with dense range, then $H$ has Property (T) \cite[Theorem
  1.3.4]{bekka_property_T}. A version of \cite[Theorem
  1.3.4]{bekka_property_T} for locally compact quantum groups has been
obtained in \cite[Theorem 5.7]{daws_skalski_viselter}. Using this, it
seems plausible that one can prove that Property (T) for quantum
groups implies unimodularity similarly to the classical case.
However, the proofs in the case of quantum groups have
proceeded differently.

To the best of our knowledge,
the first result in this direction for quantum groups was
proven for discrete quantum groups
\cite[Proposition~7]{Fima_property_T_discrete_quantum_group}. This was
subsequently 
generalised to second countable locally compact quantum groups
\cite[Section~6]{brannan_property_T}. The proof in the
case of a locally compact
quantum group, as given in \cite{brannan_property_T}, proceeds via
showing that non-unimodular locally compact quantum groups always
admit a weakly mixing representation that weakly contains the
trivial representation, as a consequence the quantum group cannot have Property (T). 

We give a very short proof of the fact that Property (T) implies
unimodularity, using a completely different technique, which does not
require the second countability assumption.  

\begin{thm}\label{thm: property T implies unimodularity}
Let $\qg$ be a locally compact quantum group with Property (T). Then
$\qg$ is unimodular.  
\end{thm}

\begin{proof}
Let $\delta$ denote the modular element of $\qg$, so that
$\delta$ is a strictly positive element affiliated to $C_0(\qg)$
\cite[Definition 7.11]{kusvaes}.
By \cite[Proposition 7.12]{kusvaes}, for all $s,t\in\IR$, 
\begin{enumerate}[label=\textup{(\roman*)}]
\item
$\Delta(\delta^{is})=\delta^{is}\ot\delta^{is}$,
\item
$\tau_t(\delta^{is})=\delta^{is}$.
\end{enumerate}
Note that $\delta^0=1$ is the trivial representation of $\qg$.
Relation (i) implies that for all $s\in\IR$, $\delta^{is}$ is a
$1$-dimensional unitary representation of $\qg$,
so there exists C*-representations
$\pi_s:C^u_0(\qgdual)\to\IC$ such that
$(\iota\ot\pi_s)(\wW)=\delta^{is}$.
For all $\omega\in B(L^2(\qg))_\ast$, it follows that
\[
\lim_{s\to0} \pi_s( (\omega\ot\iota)(\wW) )
=\lim_{s\to0}\omega(\delta^{is})
=\omega(1)=\widehat{\epsilon}_u((\omega\ot\iota)(\wW)),
\] 
where $\widehat{\epsilon}_u$ is the counit of $C_0^u(\widehat\qg)$.
Since  the family $\{\pi_s\}_{s\in\IR}$ is
uniformly bounded and
the elements of the form $(\omega\ot\iota)(\wW)$ for
$\omega\in B(L^2(\qg))_*$ are norm-dense in $C_0^u(\qgdual)$
\cite[Equation (5.2)]{kus}, it follows that
$\lim_{s\to0}\pi_s(x)=\widehat{\epsilon}_u(x)$ for all $x\in
C^u_0(\qgdual)$.

By \cite[Proposition 9.1]{kus}, 
$(\tau_t\ot\iota)(\wW)=(\iota\ot\widehat{\tau}^u_{-t})\wW$ for all
$t\in\IR$, where $(\widehat{\tau}^u_t)_{t\in\IR}$
is the universal scaling group of $\qgdual$.
Since $\qg$ has Property (T) by hypothesis,
$\widehat{\epsilon}_u$ is isolated in $\widehat{C_0^u(\qgdual)}$.
Since $\lim_{t\to0}\pi_t(x) = \widehat{\epsilon}_u(x)$
for all $x\in C_0^u(\qgdual)$, it follows that the net
$(\pi_t)_{t\in\IR}$ converges to $\widehat{\epsilon}_u$ in the Fell
topology of $\widehat{C_0^u(\qgdual)}$.
As $\widehat{\epsilon}_u$ is isolated, we must have 
$\pi_t=\widehat{\epsilon}_u$ for all $t\in\IR$
with $|t|$ sufficiently small, that is,
$\delta^{it}=1$ for all $t\in\IR$ with $|t|$ sufficiently small.
It follows that in fact $\delta^{it} = 1$ for all $t\in\IR$, and so
that $\delta = 1$, as required.
\end{proof}

\subsection{Other aspects of quantum groups with Property (T)}
Combining Proposition~\ref{Proposition: covariant finite-dimensional
  C* representations give admissible corepresentations},
Theorem~\ref{thm:admissibility-unimodular} and the proof of ``\textup{(T6)} $\implies$ \textup{(T2)}" in Theorem~\ref{thm:kazhdan} we have the following.  

\begin{crlre}\label{Corollary: T5 iff T6 for unimodular}
    Let $\qg$ be a locally compact quantum group
  that is either unimodular or arises
through the bicrossed product construction as described in
Theorem~\ref{thm:bicrossed}.
  Then $\qg$ having Property (T) is equivalent to either of the following 
  statements:
\begin{enumerate}
\item[\textup{(T6$'$)}]  
There is a finite-dimensional irreducible C*-representation of
$C_0^u(\qgdual)$ that is an isolated point in $\widehat{C^u_0(\qgdual)}$.  
\item[\textup{(T6)}] 
$C^u_0(\qgdual)\cong B\oplus M_n(\IC)$ for some C*-algebra $B$ and
  some $n\in\IN$.
\end{enumerate} 
\end{crlre}

A unitary representation of a locally compact quantum group
is \emph{weakly mixing} if it admits no nonzero admissible
finite-dimensional subrepresentation
(see \cite{viselter}).
Corollary~\ref{Corollary: T5 iff T6 for unimodular}
together with \cite[Lemma~3.6]{brannan_property_T} gives the
Bekka--Valette theorem for all unimodular locally compact quantum groups.
This was known before only for quantum groups with
trivial scaling automorphism group
\cite[Theorem~4.8]{brannan_property_T}.

\begin{crlre}\label{Corollary: Bekka-Valette theorem for unimodular
    LCQG with nontrivial scaling action}
    Let $\qg$ be a second countable locally compact quantum group
  that is either unimodular or arises
through the bicrossed product construction as described in
Theorem~\ref{thm:bicrossed}.
  Then $\qg$ has Property (T) if and only if
  every weakly mixing unitary representation of $\qg$ fails to have
  almost invariant vectors. 
\end{crlre}

Combining Corollary \ref{Corollary: Bekka-Valette theorem for
  unimodular LCQG with nontrivial scaling action} with \cite[Theorems
  3.7 \& 3.8]{brannan_property_T} we have the Kerr--Pichot theorem for
unimodular quantum groups with non-trivial scaling group.
Also this was known before only for the case of quantum groups with
trivial scaling group \cite[Theorem~4.8]{brannan_property_T}.

Given a  second countable locally compact quantum group $\qg$ and
a Hilbert space $H$, denote 
the set of all unitary representations of $\qg$ on $H$
by $\Rep(\qg, H)$
and the set of all weakly mixing unitary representations of
$\qg$ on $H$ by $W(\qg, H)$.

\begin{crlre}\label{Corollary: Kerr-Pichot theorem for unimodular LCQG
    with nontrivial scaling action} 
  Let $\qg$ be a  second countable locally compact quantum group
  that is either unimodular or arises
through the bicrossed product construction as described in
Theorem~\ref{thm:bicrossed}. Let $H$ be a separable
infinite-dimensional Hilbert space. 
If $\qg$ does not have Property (T), then
$W(\qg, H)$ is a dense $G_\delta$-set in $\Rep(\qg, H)$.  
If $\qg$ has Property (T), then $W(\qg, H)$ is closed and nowhere
dense in $\Rep(\qg, H)$. 
\end{crlre}

\end{document}